\documentclass[]{amsart}
\usepackage[utf8]{inputenc}
\usepackage{amsmath,amssymb,amsthm,mathtools}
\usepackage{graphicx}
\usepackage{verbatim}

\newtheorem{defn}{Definition}
\newtheorem{thm}{Theorem}

\newtheorem{cor}{Corollary}

\newtheorem{lem}{Lemma}

\title[Multiquadric Approximation]{Non-local approximation of continuous functions using scattered translates of the general multiquadric $(x^2+c^2)^{k-1/2}$}
\author{Jeff Ledford}
\date{June 2013}
\keywords{multiquadric, approximation, scattered data}

\begin{document}

\maketitle

\section{Introduction}
Approximation and interpolatory properties of the multiquadric have been investigated before, see for instance \cite{Baxter, Baxter Sivakumar, Powell, Riemenscheider}.  These papers deal with  integer or near-integer translates of the multiquadric $(x^2+c^2)^{1/2}$.  In \cite{Powell}, it was shown that continuous functions on a closed interval may be uniformly approximated by scattered translates of the Hardy multiquadric.  We will adapt the method found there to our purposes, showing that the same is true for the $k^{th}$ order multiquadric, $\phi_k(x)=(x^2+c^2)^{k-1/2}$, where $k\in\mathbb{N}$.  The case $k=1$ is the subject of Section 4 in \cite{Powell}.  The family of general multiquadrics has also been studied, \cite{Beatson Dyn, Guo}, although the aims of those papers are a bit different than the present goal, since they consider divided differences of the general multiquadric.

\noindent This note is organized as follows.  In the next section, various definitions and facts are collected.  The third section contains the main theorem to be proved, while the fourth section contains the details of the proof.

\section{Definitions and Basic Facts}
We will need to know what ``scattered" means.  For our purposes, we have the following definition in mind.
\begin{defn}
A sequence of real numbers, denoted $\mathcal{X}$, is said to be $\delta$-separated if
\[
\inf_{\overset{x,y\in\mathcal{X}}{x\neq y}}|x-y|= \delta >0
\]
\end{defn}
\noindent It's not hard to see that a $\delta$-separated sequence must be countable.  Take intervals of length $\delta/3$ centered at each point in $\mathcal{X}$, each of these intervals is disjoint and contains a rational number $r$.  Letting a member of $\mathcal{X}$ corrspond to the number $r$ which is in the same interval shows that the set $\mathcal{X}$ is at most countable.  This allows us to index $\mathcal{X}$ with the integers.
\begin{defn}
A sequence $\{x_j\}\subset \mathbb{R}$ is scattered if it is $\delta$-separated for some positive $\delta$ and satisfies 
\[
\lim_{j\to\pm\infty}x_j=\pm\infty
\]
\end{defn}
\noindent Throughout the remainder of the paper we let $\mathcal{X}=\{x_j\}_{j\in\mathbb{Z}}$ be a fixed but otherwise arbitrary scattered sequence.

\begin{lem}\label{sum}
For $N\in\mathbb{N}$, $0\leq l \leq N$, and $p$ a polynomial of degree $l$.  We have,
\[
\sum_{j=0}^{N}(-1)^j\binom{N}{j}p(j)=  \left\{
     \begin{array}{lr}
       0 & 0\leq l < N\\
       (-1)^Na_N\cdot N! & l=N
     \end{array}
   \right.,
\]
where $a_N$ is the leading coefficient of $p$.
\end{lem}
\begin{proof}
To see this we need only to use the binomial series expansion.  For $N\in\mathbb{N}$, we have,
\[
(1-x)^N=\sum_{j=0}^{N}(-1)^j\binom{N}{j}x^j.
\]
Now we can differentiate $l$ times to yield
\[
\left(\dfrac{d}{dx}  \right)^l(1-x)^N = \sum_{j=0}^{N}(-1)^j\binom{N}{j}j(j-1)\cdots(j-l+1)x^{j-l}.
\]
Since we can write an $l$-th degree polynomial $p(j)$ as an appropriate linear combination of 
\[
\left\{1,j, j(j-1), j(j-1)(j-2),\dots, j(j-1)\cdots(j-l+1) \right\},
\] all we must do to get the result is evaluate at $x=1$.
\end{proof}

\section{The Main Result}
\begin{thm}
Given $k\in\mathbb{N}$, a scattered sequence $\{x_j\}$, $\epsilon > 0 $, and a continuous function $f:[a,b]\to\mathbb{R}$, we may find a sequence of coefficients $\{a_j\}_{j=1}^{N}$, such that
\[
\sup_{x\in [a,b]}\left|f(x) - \sum_{j=1}^{N}a_j\phi_k(x-x_j)     \right| < \epsilon. 
\]
\end{thm}
\begin{proof}[Sketch of Proof]  The idea is to develop a Taylor expansion
\begin{equation}\label{taylor}
\phi_k(x-y) = y^{2k-1}\sum_{j=0}^{\infty}\dfrac{A_{k,j}(x)}{y^j}.
\end{equation}
Here, we will take $y>>0$, so that the series converges.  Then we show that the linear span of $\{A_{k,j}(x)\}$ contains $x^j$ for $j=0,1,2,\dots$.  We then find coefficients to approximate an $n^{th}$ degree polynomial by using an appropriate Vandermonde matrix.  Finally, since we may approximate polynomials, we appeal to the Stone-Weierstrass Theorem to finish the proof.
\end{proof}
\noindent This theorem, when combined with H\"{o}lder's Inequality allows us to replace the $L^{\infty}([a,b])$ norm above with the $L^p([a,b])$ norm.  We state this in the following corollary.
\begin{cor}
Given $k\in\mathbb{N}$, a scattered sequence $\{x_j\}$, $\epsilon > 0 $, and a continuous function $f:[a,b]\to\mathbb{R}$, we may find a sequence of coefficients $\{a_j\}_{j=1}^{N}$, such that
\[
\left\|f(x) - \sum_{j=1}^{N}a_j\phi_k(x-x_j)     \right\|_{L^p([a,b])} < \epsilon.
\]
\end{cor}

\section{Details}
This section provides a rigorous justification for the outline of the proof. We begin with the Taylor expansion, which we recognize as the familiar binomial series.  For $y>>0$ we have,

\begin{align*}
&y^{-2k+1}\phi_k(x-y) = \\
&\sum_{n=0}^{\infty}\binom{k-\frac{1}{2}}{n}\sum_{j=0}^{n}\binom{n}{j}(-2x)^{n-j}\sum_{l=0}^{j}\binom{j}{l}x^{2(j-l)}c^{2l}y^{-(n+j)}\\
=&\sum_{n=0}^{\infty}\sum_{j=0}^{n}\sum_{l=0}^{j}\binom{k-\frac{1}{2}}{n}\binom{n}{j}\binom{j}{l}c^{2l}(-2)^{n-j}\dfrac{x^{j+n-2l}}{y^{n+j}}\\
=&\sum_{n=0}^{\infty}\sum_{j=n}^{2n}\sum_{l=0}^{j-n}\binom{k-\frac{1}{2}}{n}\binom{n}{j-n}\binom{j-n}{l}c^{2l}(-2)^{2n-j}\dfrac{x^{j-2l}}{y^{j}}\\
=&\sum_{j=0}^{\infty}\dfrac{(-1)^j}{y^{j}}\left\{\sum_{n=\lceil j/2 \rceil}^{j}\sum_{l=0}^{j-n}\binom{k-\frac{1}{2}}{n}\binom{n}{j-n}\binom{j-n}{l}c^{2l}2^{2n-j} x^{j-2l}   \right\} \\
=&\sum_{j=0}^{\infty}\dfrac{A_{k,j}(x)}{y^j}.
\end{align*}
In the fourth line we have re-indexed the sum, and in the fifth changed the order of summation.  All of this hinges on the binomial series being absolutely convergent, but since we've assumed $y>>0$, the argument of the binomial series will be close to $0$.  This gives us a formula for the polynomials $A_{k,j}(x)$:
\begin{equation}\label{poly}
A_{k,j}(x)=(-1)^{j}\sum_{n=\lceil j/2 \rceil}^{j}\sum_{l=0}^{j-n}\binom{k-\frac{1}{2}}{n}\binom{n}{j-n}\binom{j-n}{l}2^{2n-j}c^{2l}x^{j-2l}.
\end{equation}
We can glean lots of information from \eqref{poly}, for instance, $\deg(A_{k,j})$ has the same parity as $j$ and $\deg(A_{k,j})\leq j$.  We are interested in the leading term for $A_{k,j}$, for this will tell us the exact degree.  We need only re-index and change the order of summation, since both sums are finite, there is no problem with convergence.
\begin{align*}
&\sum_{n=\lceil j/2 \rceil}^{j}\sum_{l=0}^{j-n}\binom{k-\frac{1}{2}}{n}\binom{n}{j-n}\binom{j-n}{l}2^{2n-j}c^{2l}x^{j-2l}\\
=&\sum_{l=0}^{\lceil j/2 \rceil}\sum_{n=\lceil j/2 \rceil}^{j-l}\binom{k-\frac{1}{2}}{n}\binom{n}{j-n}\binom{j-n}{l}2^{2n-j}c^{2l}x^{j-2l}\\
=&\sum_{l=0}^{\lceil j/2 \rceil} x^{j-2l} \left\{ \sum_{n=\lceil j/2 \rceil}^{j-l}\binom{k-\frac{1}{2}}{n}\binom{n}{j-n}\binom{j-n}{l}2^{2n-j}c^{2l}  \right\}
\end{align*}
To simplify notation, we write
\begin{align}\label{coeffs}
\nonumber A_{k,j}(x)&=(-1)^j\sum_{l=0}^{\lceil j/2 \rceil}c^{2l} x^{j-2l} \left\{ \sum_{n=\lceil j/2 \rceil}^{j-l}\binom{k-\frac{1}{2}}{n}\binom{n}{j-n}\binom{j-n}{l}2^{2n-j} \right\} \\
&=(-1)^j\sum_{l=0}^{\lceil j/2 \rceil}a_{j-2l}x^{j-2l}.
\end{align}
We are in position to state the following lemma.
\begin{lem}
Given $k\in\mathbb{N}$, and $j\geq 2k$, we have
\[
a_{j-2l}= \left\{
     \begin{array}{lr}
       0 & 0\leq l < k, \\
       c^{2k}\binom{k-1/2}{k} & l=k.
     \end{array}
   \right.
\]
Hence, the polynomial $A_{k,j}(x)$ is a polynomial of degree $j-2k$.
\end{lem}
\begin{proof}
We need only find the sum in \eqref{coeffs}.  To do this, we re-index.
\begin{align*}
 &\sum_{n=\lceil j/2 \rceil}^{j-l}\binom{k-1/2}{n}\binom{n}{j-n}\binom{j-n}{l}2^{2n-j} \\
=& \sum_{n=0}^{j-\lceil j/2 \rceil - l}\binom{k-1/2}{n+\lceil j/2 \rceil} \binom{n+\lceil j/2 \rceil}{j-\lceil j/2 \rceil-n} \binom{j-\lceil j/2 \rceil-n}{l} 2^{2n+ 2\lceil j/2\rceil -j }
\end{align*}
Now we let $j=2k+N$, for $N=0,1,2,\dots$, and we have two cases, the case that $N$ is even, and the case that $N$ is odd.  Both cases being similar calculations, we will work the odd case here.  By letting $N=2m+1$, we have
\begin{align*}
& \sum_{n=0}^{j-\lceil j/2 \rceil - l}\binom{k-1/2}{n+\lceil j/2 \rceil} \binom{n+\lceil j/2 \rceil}{j-\lceil j/2 \rceil-n} \binom{j-\lceil j/2 \rceil-n}{l} 2^{2n+ 2\lceil j/2\rceil -j }\\
=&\sum_{n=0}^{k+m-l}\binom{k-1/2}{n+k+m+1}\binom{n+k+m+1}{k+m-n}\binom{k+m-n}{l}2^{2n+1}\\
=&k!\binom{k-1/2}{k}\sum_{n=0}^{k+m-l}2^{n-m}(-1)^{n+m+1} \dfrac{(2(n+m)+1)!!}{(2n+1)!l!(k+m-n-l)!}\\
=&\dfrac{(-1)^{m+1}k!}{l!(k+m-l)!}\binom{k-1/2}{k}\sum_{n=0}^{k+m-l}(-1)^{n}\binom{k+m-l}{n}\dfrac{(2(n+m)+1)!!}{2^m(2n+1)!!}.
\end{align*}
The last summand may be reduced by noting that 
\[
\dfrac{(2(n+m)+1)!!}{2^m(2n+1)!!}=\dfrac{(2n+2m+1)(2n+2m-1)\cdots(2n+3)}{2^m}
\]
is a monic, $m^{th}$ degree polynomial in the variable $n$.  Thus Lemma \ref{sum}, gives us the result provided $k-l \geq 0$.  This proves the case when N is odd, the even case is virtually the same computation.
\end{proof}
Now we choose a subset of $\mathcal{X}$ which allows us to recover $A_{k,2k+N}(x)$.  Pick a set $\{y_j:1\leq j \leq 2k+N+1\}\subset\mathcal{X}$ using the following conditions
\begin{itemize}
\item $y_1>>0$,
\item $y_j\geq 2 y_{j-1}$;\quad $j=2,3,\dots,2k+N+1$.
\end{itemize}
The modified $(2k+N+1)\times (2k+N+1)$ Vandermonde system
\[
\sum_{j=1}^{2k+N+1}b_j y_j^{l}=\delta_{l,-N-1} \qquad l=2k-1,2k-2,\dots,-N-1
\]
is invertible.  The solution may be found by repeated use of Cramer's rule  We get
\[
b_j=(-1)^{j+1}y_j^{N+1}\prod_{l=1,l\neq j}^{2k+N+1}\left[1-\dfrac{y_j}{y_l}  \right]^{-1}.
\]
As a result of this, we have that the set of products
\[
\left \{b_j y_j^{-(N+1)}:j=1,2,\dots,2K+N+1  \right\}
\]
is uniformly bounded.  Using these coefficients, we have
\begin{equation}\label{k approx}
\sum_{j=1}^{2k+N+1}b_j\phi_k(x-y_j)=A_{k,N+2k}(x)+ \mathcal{O}\left(\dfrac{1}{y_1}\right), \qquad a\leq x \leq b.
\end{equation}
Since $y_1$ may be as large as we like and $A_{k,2k+N}(x)$ is an $N^{th}$ degree polynomial, to approximate continuous functions, we need only find a polynomial to approximate on $[a,b]$, then approximate the polynomial with the sum above.

\end{document}